\documentclass[10pt]{article}
\font\smallit=cmti10
\font\smalltt=cmtt10

\usepackage{caption}
\usepackage{color}
\usepackage{amssymb,latexsym,amsmath,epsfig,amsthm} 
\usepackage{empheq}
\usepackage{here}
\usepackage{ascmac}
\makeatletter
\usepackage{here}
\usepackage{comment}

\renewcommand\section{\@startsection {section}{1}{\z@}
{-30pt \@plus -1ex \@minus -.2ex}
{2.3ex \@plus.2ex}
{\normalfont\normalsize\bfseries\boldmath}}

\renewcommand\subsection{\@startsection{subsection}{2}{\z@}
{-3.25ex\@plus -1ex \@minus -.2ex}
{1.5ex \@plus .2ex}
{\normalfont\normalsize\bfseries\boldmath}}

\renewcommand{\@seccntformat}[1]{\csname the#1\endcsname. }
\makeatother
\newtheorem{theorem}{Theorem}
\newtheorem{lemma}{Lemma}
\newtheorem{conjecture}{Conjecture}

\theoremstyle{definition}
\newtheorem{definition}{Definition}

\begin{document}

\begin{center}
\uppercase{\bf An Amalgamation Nim with Restriction}
\vskip 20pt

{\bf Hikaru Manabe}\\
{\smallit Keimei Gakuin Junior and High School, Kobe City, Japan}\\
{\tt urakihebanam@gmail.com}


\end{center}
\vskip 20pt
\centerline{\smallit Received: , Revised: , Accepted: , Published: } 
\vskip 30pt


\centerline{\bf Abstract}
\noindent
In an amalgamation Nim, players are allowed to use a move from the
traditional form of Nim or to amalgamate two piles when they are not empty. No formula that describes the set of P-positions of Amalgamation Nim is known. The author gives a condition on the amalgamation of two piles. Players can amalgamate
two piles only when the number of stones in these two piles is equal to or more than 2. Then, we get a formula that describes the set of P-positions for this variant of amalgamation Nim.
\pagestyle{myheadings} 
\markright{\smalltt    \hfill} 
\thispagestyle{empty} 
\baselineskip=12.875pt 
\vskip 30pt

\section{Introduction}
We denoted the sets of non-negative integers and natural numbers by $\mathbb{Z}_{\ge 0}$ and  $\mathbb{N}$, respectively. 
The winning strategy for the game of Nim was determined in 1902 by Bouton \cite{bouton}. Since then, there have been many variants of Nim. Nim and its variants belong to combinatorial games. For a general introduction to combinatorial games, see \cite{lesson}. 

Amalgamation Nim is a relatively new variant of Nim. One origin of Amalgamation Nim was a query \cite{richard} from a student of S.C. Locke, who is one of the authors of \cite{amalnim}. 

In this article, we Maximum Amalgamation Nim with restrictions, which is a variant of Amalgamation Nim.

In Definition \ref{defofalma}, we define Amalgamation Nim. 
\begin{definition}\label{defofalma}
The two players begin with several piles of stones, and then alternate moves, with the player who makes the last legal move being the winner. The moves are of two types: \\
$(i)$ A player chooses a pile, and removes any number of stones of the pile;\\
$(ii)$ A player can replace two non-empty piles by a single pile containing the total number of stones in the original two piles. 
\end{definition}

We deal with impartial games with no draws; therefore, there are only two outcome classes. 

\begin{definition}
$(a)$ A position is called a $\mathcal{P}$-\textit{position} if it is a winning position for the previous player (the player who just moved), as long as he/she plays correctly at every stage.\\
$(b)$ A position is called an $\mathcal{N}$-\textit{position} if it is a winning position for the next player, as long as he/she plays correctly at every stage. 
\end{definition}

 \begin{definition}\label{defofmoveg}
	For any position $\mathbf{p}$ in game $\mathbf{G}$, a set of positions can be reached by a single move in $\mathbf{G}$, which we denote as \textit{move}$(\mathbf{p})$. 
\end{definition}

One of the most important topics of impartial games is to determine the set of $\mathcal{P}$-position. 

We denote a position of the game by $(x,y,z)$, where $x,y,z$ are the number of stones in the first, the second, and the third pile. 

\begin{theorem}
For two-pile game of Amalgamation Nim in Definition \ref{defofalma}, the set of  $\mathcal{P}$-positions is $\{(x,y):x=y\}$. Therefore,  Amalgamation Nim is the same as traditional Nim for two-pile game.
\end{theorem}

When the number of piles is larger than two, there is no formula that describes  $\mathcal{P}$-positions of the game in Definition
\ref{defofalma}. However, in \cite{amalnim}, there is a list of
P-positions for the three-heap Nim with one heap of height at most seven.

\section{ An Amalgamation Nim with Restriction}
We define an Amalgamation Nim with Restriction. This is a variant of Amalgamation Nim in which there is a restriction on the number of stones when we merge two piles.
\begin{definition}\label{defofamanimris}
Two players take turns to do one of the following $(i)$ and $(ii)$.\\
$(i)$ A player chooses a pile and removes any number of stones from the pile.\\
$(ii)$ A player can replace two piles with a single pile containing the total number of stones in the original two piles when the numbers of stones in two piles are equal to or more than 2. 
\end{definition}

\begin{definition}\label{moveofalmarest}
	For $x,y,z \in \mathbb{Z}_{\ge 0}$, we define \\
\begin{align}
& move(x,y,z)= \{(u,y,z):u \in \mathbb{Z}_{\geq0} \text{ and } u<x \}\label{movex} \\
& \  \ =\{(x,v,z):v \in \mathbb{Z}_{\geq0} \text{ and } v <y\}\label{movey} \\
& \  \ =\{(x,y,w):w \in \mathbb{Z}_{\geq0} \text{ and } w <z\} \label{movez} \\
& \ \  =\{(x+y,0,z),(0,x+y,z):\text{ if } x,y \geq 2\} \label{movemerge1} \\
& \ \  =\{(x,y+z,0),(x,0,y+z):\text{ if } y,z \geq 2\} \label{movemerge2} \\
& \ \  =\{(x+z,y,0),(0,y,x+z):\text{ if } x,z \geq 2\}. \label{movemerge3}
\end{align}
\end{definition}

For numbers $x,y,z,u,v,w \in \mathbb{Z}_{\ge 0}$,
let $x=\sum_{i=0}^{n}x_i$, $y=\sum_{i=0}^{n}y_i$, $z=\sum_{i=0}^{n}z_i$, $u=\sum_{i=0}^{n}u_i$, $v=\sum_{i=0}^{n}v_i$, $w=\sum_{i=0}^{n}w_i$, where $x_i, y_i, z_i, u_i, v_i, w_i \in \{0,1\}.$

\begin{definition}\label{defofpn}
We define sets as the followings. \\
$(i.1)$ $P_{0,1}=\{(x,y,z):x+y=z, x,y \leq z, x \oplus y \oplus z = 0
\text{ and } \min(x,y) < 2 \text{ for } x,y,z \in \mathbb{Z}_{\ge 0}\}$.\\
$(i.2)$ $P_{0,2}=\{(x,y,z):x+y>z+2,  x,y \leq z, \text{ and } x \oplus y \oplus z = 0 \text{ for } x,y,z \in \mathbb{Z}_{\ge 0}\}$. \\
$(ii.1)$ $N_{0,1}=\{(x,y,z):x+y=z, x,y \leq z, x \oplus y \oplus z = 0
\text{ and } x,y \geq 2 \text{ for } x,y,z \in \mathbb{Z}_{\ge 0}\}$.\\
$(ii.2)$ $N_{0,2}=\{(x,y,z):x+y=z+2,  x,y \leq z, \text{ and } x \oplus y \oplus z = 0 \text{ for } x,y,z \in \mathbb{Z}_{\ge 0}\}$.\\
$(iii)$ $P_{1,1}=\{(x,y,z+1):x+y \text{ is } even \text{ and } (x,y,z) \in N_{0,1}\}$
$ \cup \{(x,y,z-1):x+y \text{ is } odd \text{ and } (x,y,z) \in N_{0,1}\}$.\\
$(iv)$ $P_{1,2}=\{(x,y,z+1):x+y \text{ is } even \text{ and } (x,y,z) \in N_{0,2}\}$
$ \cup \{(x,y,z-1):x+y \text{ is } odd \text{ and } (x,y,z) \in N_{0,2}\}$.\\
$(v)$ $P_0=\{(x,y,z):(x,y,z) \in P_{0,1} \cup P_{0,2}\}$
$\cup \{(x,y,z):(x,z,y) \in P_{0,1} \cup P_{0,2}\}$
$\cup \{(x,y,z):(y,z,x) \in P_{0,1} \cup P_{0,2}\}$.\\
$(vi)$ $P_1=\{(x,y,z):(x,y,z) \in P_{1,1} \cup P_{1,2}\}$
$\cup \{(x,y,z):(x,z,y) \in P_{1,1} \cup P_{1,2}\}$
$\cup \{(x,y,z):(y,z,x) \in P_{1,1} \cup P_{1,2}\}$.\\
$(vii)$ $N_0=\{(x,y,z):(x,y,z) \in N_{0,1} \cup N_{0,2}\}$
$\cup \{(x,y,z):(x,z,y) \in N_{0,1} \cup N_{0,2}\}$
$\cup \{(x,y,z):(y,z,x) \in N_{0,1} \cup N_{0,2}\}$.\\
$(viii)$ $P=P_0 \cup P_1 $.
\end{definition}

\begin{lemma}\label{lemmaforxyz}
Suppose that 
\begin{equation}
x \oplus y \oplus z = 0. \label{xoplusyoplusz0} 
\end{equation}
Then, we have the followings.\\
$(i)$  $x+y=z$ if and only if $x_i+y_i=z_i$ for $i=0, \dots, n$.\\
$(ii)$  $x+y=z+2$ if and only if $x_i+y_i=z_i$ for $i=1, \dots, n$, $x_0=y_0=1$ and $z_0=0$.\\
$(iii)$ $x+y>z+2$ if and only if  $x_j=y_j=1$ and $z_j=0$ for some $j\geq 1$.\\
\end{lemma}
\begin{proof}
By (\ref{xoplusyoplusz0}), for each $i=0,1,2,\dots,n$, we have
\begin{equation}
x_i+y_i=z_i    \label{xiyiequzi}
\end{equation}
or
\begin{equation}
x_i=y_i=1  \text{ and }  z_i=0. \label{xieqyi1andzi0}
\end{equation}
If we have (\ref{xieqyi1andzi0}) for some $i \in \mathbb{Z}_{\ge 0}$, $x+y \geq z + 2 \times 2^i$.
Therefore, we have $(i)$, $(ii)$ and $(iii)$.
\end{proof}

\begin{lemma}\label{lemmaforn02}
$(i)$ For $x,y,z \in N_{0,1}$, we have $x,y \geq 2$,
$z\geq 4$, $z \geq x+2, y+2$ and $x+y=z$.\\
$(ii)$ 
For $(x,y,z) \in N_{0,2}$, we have $x,y \geq 3$, $z \geq x+1, y+1$, $x+y=z+2$, $x,y$ are odd and $z$ is even.\\
$(iii)$ For $(x,y,z) \in P_{1,1}$, we have $x,y \geq 2$, $z \geq 3$, $z \geq x+1,y+1$ and $x+y = z+1$ or $x+y=z-1$.\\
$(iv)$ For $(x,y,z) \in P_{1,2}$, we have 
$x,y,z \geq 3$, $z \geq x+2, y+2$, and 
$x+y=z+1$.\\
$(v)$ For $(x,y,z) \in P_{0,1}$, we have 
$x,y,z \geq 1$ or $(x,y,z) \in \{(0,k,k):k \in \mathbb{Z}_{\ge 0}\} \cup \{(k,0,k):k \in \mathbb{Z}_{\ge 0}\}.$  $x+y+z$ is even.
\\
$(vi)$ For $(x,y,z) \in P_{0,2}$, we have 
$x,y,z \geq 1$ and $x+y+z$ is even.
\end{lemma}
\begin{proof}
$\mathrm{(i)}$
For $(x,y,z) \in N_{0,1}$, 
by (ii.1) of  Definition \ref{defofpn},
\begin{equation}
x + y = z \label{xyequaltoz}
\end{equation}
and 
\begin{equation}
x, y \geq 2.  \label{xylargerthan2}  
\end{equation}
By (\ref{xyequaltoz}) and (\ref{xylargerthan2} ), 
\begin{equation}
z \geq x+y \geq x+2, y+2  
\end{equation}
and
\begin{equation}
z \geq 4.  
\end{equation}
$\mathrm{(ii)}$
Let $(x,y,z) \in N_{0,2}$. Then, by (ii.2) of  Definition \ref{defofpn},
\begin{equation}
x \oplus y \oplus z = 0 \label{xyz0n02}
\end{equation}
and
\begin{equation}
x + y = z+2.  \label{xyequalzplus2}
\end{equation}
By (\ref{xyz0n02}), (\ref{xyequalzplus2}), and $(ii)$ of Lemma \ref{lemmaforxyz},
\begin{equation}
x_0, y_0 =1  \text{ and }  z_0=0.  \label{x0y0eq1andz00}
\end{equation}
By (\ref{x0y0eq1andz00}), $x,y$ are odd and $z$ is odd.
Since $x \leq z$, by (\ref{lemmaforxyz} and (\ref{x0y0eq1andz00}), there exists $i \geq 1$ such that $z_i = 1$ and $x_i=0$. By (\ref{xyz0n02}), $y_i=1$. 
Since $y_0=1$, $y_i=1=z_i$, and $y \leq z$, there exists $j \geq 1$ such that $z_j = 1$ and $y_j=0$. By (\ref{xyz0n02}), $x_j=1$. Therefore, we have 
\begin{equation}
x,y \geq 3   \nonumber
\end{equation}
and
\begin{equation}
z=x+y-2 \geq x+1, y+1.   \nonumber
\end{equation}
$\mathrm{(iii)}$ 
Let $(x,y,z) \in P_{1,1}$. 
then there exists $w$ such that
$(x,y,w) \in N_{0,1}$, $z = w+1$ or $z=w-1$.
By (\ref{xyequalzplus2}) and $(i)$ of this lemma,
\begin{equation}
x+y = w = z+1 \text{ or} z-1. \label{xplusyleqzp1}
\end{equation}
By $(i)$, $x,y \geq 2$ and
\begin{equation}
z \geq w-1 \geq x+1,y+1.   \nonumber
\end{equation}
$\mathrm{(iv)}$  For $(x,y,z) \in P_{1,2}$, there exists $w$ such that 
$z = w-1$ or $z = w+1$ and 
$(x,y,w) \in N_{0,2}$.
Then,
\begin{equation}
x+y = w+2. \label{uvwprime2}
\end{equation}
By $(ii)$, 
\begin{equation}
x,y \geq 3,\label{xylarge3}
\end{equation}
$x,y$ are odd, and $x+y$ is even. Hence we have 
$z=w+1$, and by (\ref{uvwprime2}) and (\ref{xylarge3}),
\begin{equation}
z = w+1 = x+y -1 \geq x+2, y+2.  \label{uvwprime3}
\end{equation}
By (\ref{uvwprime2}),
$x+y=z+1$.\\
$\mathrm{(v)}$ Let $(x,y,z) \in P_{0,1}$. Then, by (i.1) of  Definition \ref{defofpn},
\begin{equation}
x \oplus y \oplus z = 0\label{xyzoplus0a}
\end{equation}
and
\begin{equation}
x+y=z \text{ and } x,y \leq z.\label{xpyeqalzxyz}
\end{equation}
Hence, if $z=0$, by (\ref{xpyeqalzxyz}) we have $x=y=z=0$.
If $x=0$, by (\ref{xpyeqalzxyz}) we have $y=z$.
If $y=0$, by (\ref{xpyeqalzxyz}) we have $x=z$.
By (\ref{xyzoplus0a}), $x+y+z$ is even.\\
$\mathrm{(vi)}$ Let $(x,y,z) \in P_{0,2}$. Then, by (i.1) of  Definition \ref{defofpn},
\begin{equation}
x \oplus y \oplus z = 0\label{xyzoplus0}
\end{equation}
and
\begin{equation}
x+y>z+2 \text{ and } x,y \leq z.\label{xpyeqalzp2}
\end{equation}
If $z=0$, by (\ref{xpyeqalzp2}) we have $x=y=z=0$. This contradicts (\ref{xpyeqalzp2}).
If $x=0$, by (\ref{xyzoplus0}) we have $y=z$.  This contradicts (\ref{xpyeqalzp2}).
If $y=0$ by (\ref{xyzoplus0}) we have $x=z$.  This contradicts (\ref{xpyeqalzp2}).
By (\ref{xyzoplus0}), $x+y+z$ is even.\\
\end{proof}

\begin{lemma}\label{lemmafp1andp0}
$(i)$ Suppose that we move from $(x,y,z) \in P_{0}$ to $(u,v,w) 
 \in P_1$ without using amalgamation. Then we have one of the followings:\\
 $(a)$  $x=u+1$, $y=v$, $z=w$, and $x$ is odd;\\
 $(b)$   $x=u$, $y=v+1$, $w=z$, and $y$ is odd;\\
 $(c)$   $x=u$, $y=v$, $z=w+1$, and $z$ is odd.\\
$(ii)$ Suppose that we move from $(u,v,w) \in P_{1}$ to $(x,y,z) 
 \in P_0$ without using amalgamation. Then we have one of the followings:\\
$(a)$  $u=x+1$,$v=y$, $w=z$ and $u$ is odd;\\
 $(b)$   $u=x$,$v=y+1$, $w=z$  and $v$ is odd;\\
 $(c)$   $x=u$,$y=v$, $w=z+1$  and $w$ is odd.
\end{lemma}
\begin{proof}
Suppose that $(x,y,z) \in P_0$ and  $(u,v,w) \in P_1$ . Then,
\begin{equation}
x \oplus y \oplus z = 0 \label{xyz0}
\end{equation}
and 
\begin{equation}
u \oplus v \oplus w = 1. \label{uvw1}
\end{equation}
$\mathrm{(i)}$
It is sufficient to prove the case that $u<x$, $v=y$ and $w=z$. 
Suppose that $u_i=x_i$ for $i = n,n-1, \dots, j+1$ and $u_j=0 < 1=x_j$. If $j \geq 1$, then 
$u \oplus v \oplus w \geq 2$, and this contradicts (\ref{uvw1}).
Therefore, $j=0$, $u=x-1$ and $x$ is odd.\\
$\mathrm{(ii)}$ It is sufficient to prove the case that $x<u$, $y=v$ and $z=w$.
Suppose that $x_i=u_i$ for $i = n,n-1, \dots, j+1$ and $x_j=0 < 1=u_j$. If $j \geq 1$, then 
$x \oplus y \oplus z \geq 2$, and this contradicts (\ref{xyz0}).
Therefore, $j=0$, $x=u-1$ and $u$ is odd.\\
\end{proof}

\begin{lemma}\label{lemmafp1andp0b}
$(i)$ Suppose that we move from $(x,y,z) \in P_{0}$ to $(u,v,w) 
 \in P_1$ without using amalgamation. Then,
 \begin{equation}
|x-u|+|y-v|+|z-x| \leq 1.     
 \end{equation}
$(ii)$ Suppose that we move from $(x,y,z) \in P_1$ to $(u,v,w) 
 \in P_0$ without using amalgamation. Then,
 \begin{equation}
|x-u|+|y-v|+|z-x| \leq 1.     
 \end{equation}
 \end{lemma}
\begin{proof}
$(i)$ and $(ii)$ are direct from Lemma \ref{lemmafp1andp0}.
\end{proof}

\begin{lemma}\label{p1notp0}
Suppose that  we start with a position  $(x,y,z) \in P_1$. Then, we cannot move to a position in $P_0$.
\end{lemma}
\begin{proof}
$[I]$ Suppose that we use amalgamation. If $(x,y,z) \in P_{1,1} \cup P_{1,2}$, then by 
$(iii)$ and $(iv)$ of Lemma \ref{lemmaforn02}, 
$x+y=z+1$ or $x+y=z-1$. Then,
$x+y+z=2z+1$ or $x+y+z=2z-1$, and $x+y+z$ is odd.
For any $(u,v,w) \in P_0$, by $(v)$ and $(vi)$ of Lemma \ref{lemmaforn02}, 
$u+v+w$ is even.
By using amalgamation, the total number of stones in three piles will not change.
Hence we cannot move from $(x,y,z)$ to $(u,v,w)$.\\
$[II]$
Suppose that we start with $(x,y,z) \in P_{1,1}$ and we move to 
\begin{equation}
(u,v,w) \in P_{0} \label{uvwinp0}
\end{equation}
 without using amalgamation. 
By $(iii)$ of Lemma \ref{lemmaforn02}, 
\begin{equation}
x,y,z \geq 2,\label{conditiona}
\end{equation}
\begin{equation}
x,y \leq z-1 \label{conditionb}
\end{equation}
and
\begin{equation}
x+y \leq z+1 \label{conditione}
\end{equation}
By (\ref{conditiona}),
there exist $i,j,k \in \mathbb{N}$ such that 
\begin{equation}
x_i=y_j=z_k=1. \label{xyz1a}
\end{equation}
By (ii) of Lemma \ref{lemmafp1andp0b} and (\ref{conditionb}),
\begin{equation}
u,v \leq w, \nonumber
\end{equation}
and hence by (\ref{uvwinp0}) we have 
\begin{equation}
(u,v,w) \in P_{0,1} \cup P_{0,2}. \label{uvwinp1p02}
\end{equation}
By (\ref{xyz1a}) and (ii) of Lemma \ref{lemmafp1andp0b},
\begin{equation}
u,v,w \geq 2. \label{conditiond}
\end{equation}
By (i.1) of Definition \ref{defofpn} and (\ref{conditiond}),
\begin{equation}
(u,v,w) \notin P_{0,1}.\nonumber
\end{equation}
By (\ref{conditione}) and (ii) of Lemma \ref{lemmafp1andp0b},
\begin{equation}
u+v \leq w+2,
\end{equation}
and this contradicts (i.2) of Definition \ref{defofpn}.\\
$[III]$
We suppose that we start with 
\begin{equation}
(x,y,z) \in P_{1,2}
\end{equation}
and  we move to 
\begin{equation}
(u,v,w) \in P_{0} \label{uvwinp0in}
\end{equation}
without using amalgamation.
By (iv) of Lemma \ref{lemmaforn02},
\begin{equation}
x,y,z \geq 3,   \label{p12top01}
\end{equation}
\begin{equation}
z \geq x+2, y+2, \label{p12top02}  
\end{equation}
and 
\begin{equation}
x+y=z+1.  \label{p12top03}
\end{equation}
By (\ref{p12top02}) and (ii) of Lemma \ref{lemmafp1andp0b}
\begin{equation}
w \geq u, v. \label{p12top02wlar}  
\end{equation}
Hence by (\ref{uvwinp0in}),
\begin{equation}
(u,v,w) \in P_{0,1}\cup P_{0,2}. \label{uvwinp0in2}
\end{equation}
$[III.1]$
By (\ref{p12top01}) and and (ii) of Lemma \ref{lemmafp1andp0b},
\begin{equation}
u,v,w \geq 2.
\end{equation}
Hence, by (i) of Definition \ref{defofpn},
$(u,v,w) \notin P_{0,1}.$\\
$[III.2]$ Suppose that $(u,v,w) \in P_{0,2}.$
Then, by (i.2) of Definition \ref{defofpn}  
\begin{equation}
u+v > w+2.\label{uvlargerthanw2}    
\end{equation}
(\ref{uvlargerthanw2} ) and (\ref{p12top03})
contradict (ii) of Lemma \ref{lemmafp1andp0b}.
\end{proof}

\begin{lemma}\label{p0notp0}
Suppose that we start with a position  $(x,y,z) \in P_0$. Then, we cannot move to a position in $P_0$.
\end{lemma}
\begin{proof}
We suppose that and we move to 
$(u,v,w) \in P_0$.
    Since $(x,y,z), (u,v,w) \in P_0$, we have
\begin{equation}
x \oplus y \oplus z =0   \label{nimsum01} 
\end{equation}
and
\begin{equation}
u \oplus v \oplus w =0.  \label{nimsum02} 
\end{equation}  
$[II]$ If we move to $(u,v,w)$ without using amalgamation, i.e.,(\ref{movex}), (\ref{movex}), (\ref{movex}) of Definition \ref{moveofalmarest}, a move from $(x,y,z)$ to $(u,v,w)$ that satisfies (\ref{nimsum01}) and (\ref{nimsum02}) is impossible.\\
$[II]$ Suppose that we move to $(u,v,w)$ by using amalgamation.
By We assume without any loss of generality,
$(u,v,w)=(x+y,0,z)\in P_0$.
By (\ref{nimsum02}), 
\begin{equation}
    x+y=z,\label{xplusyequalz}
\end{equation}
and hence
\begin{equation}
x,y \leq z.\label{xyzorder}
\end{equation}
Then, we have $(x,y,z) \in  P_{0,1} \cup  P_{0,2}.$\\
$[II.1]$ If $(x,y,z) \in  P_{0,1}$, by (i.1) of Definition \ref{defofpn} 
\begin{equation}
\min(x,y) <2. \label{minxyless2}
\end{equation}
(\ref{minxyless2}) contradicts the use of amalgamation of $x,y$.\\
$[II.1]$
If  $(x,y,z) \in  P_{0,2}$, by (i.2) of Definition \ref{defofpn} 
\begin{equation}
x+y > z+2. \label{xplusylargez2}
\end{equation}
(\ref{xplusylargez2}) contradicts (\ref{xplusyequalz}).
\end{proof}

\begin{lemma}\label{p1notp1}
Suppose that we start with a position  $(x,y,z) \in P_1$. Then, we cannot move to a position in $P_1$.
\end{lemma}
\begin{proof} 
Suppose that we start with a position  $(x,y,z) \in P_1$ and move to a position  $(u,v,w) \in P_1$.
Then, we have 
\begin{equation}
x \oplus y \oplus z = 1  \label{xyzsum1}
\end{equation}
and
\begin{equation}
u \oplus v \oplus w = 1  \label{uvwsum1}
\end{equation}
$[I]$ 
Suppose that we  move to $(u,v,w)$ by the usual operation of Nim, i.e.,(\ref{movex}), (\ref{movex}), (\ref{movex}) of Definition \ref{moveofalmarest}, 
(\ref{xyzsum1}) contradicts (\ref{uvwsum1}).\\
$[II]$ 
Since it is not possible by the usual operation of Nim to move from a position of Nim-sum 1 to a position to a Nim-sum 1, we assume that we use a merge of two piles.
Then, you get a pile of no stone. By Lemma \ref{lemmafp1andp0b}, for any position $(u,v,w) \in P_1$, $u,v,w \geq 1$, and hence 
you can not move to a position in $P_1$ by the merge of two piles. Therefore we can finish the proof of this lemma.
\end{proof}

\begin{lemma}\label{p01notp1}
Suppose that we start with a position  $(x,y,z) \in P_{0,1}$. Then, we cannot move to a position in $P_{1}$.
\end{lemma}
\begin{proof}
Suppose that we move to $(u,v,w)\in P_1$. 
By $(iii)$ and $(iv)$ of Lemma \ref{lemmaforn02}, 
\begin{equation}
u,v,w \geq 2. \label{uvwlarge2b}
\end{equation}
$[I]$ Suppose that we use amalgamation. Then,
we have $u=0$ or $v=0$ or $w=0$. These contradict $(\ref{uvwlarge2b})$.\\
$[II]$ Suppose that we do not use amalgamation.
Since $(x,y,z) \in P_{0,1}$,
\begin{equation}
\min(x,y) <2. \label{minxy2}
\end{equation}
Inequality (\ref{uvwlarge2b}) contradicts Inequality (\ref{minxy2}).
\end{proof}

\begin{lemma}\label{p02notp1}
Suppose that we start with a position  $(x,y,z) \in P_{0,2}$. Then, we cannot move to a position in $P_{1}$.
\end{lemma}
\begin{proof}
Suppose that we start with a position  $(x,y,z) \in P_{0,2}$.
By Lemma \ref{lemmaforxyz}, there exists $j \geq 1$ such that 
\begin{equation}
x_j=y_j=1 \label{xjyj1}
\end{equation}
and
\begin{equation}
z_j=0.\label{zj0} 
\end{equation}
Then,  
\begin{equation}
x+y \geq z+4.\label{xyz4}
\end{equation}
Since $x,y \leq z$, we have 
\begin{equation}
y+z \geq x+4\label{xyz5}
\end{equation}
and
\begin{equation}
z+x \geq y+4.\label{xyz6}
\end{equation}
If we move from $(x,y,z)$ to $(u,v,w) \in P_1$, then
by (i) of Lemma \ref{lemmafp1andp0b},
(\ref{xyz4}), (\ref{xyz5}), and (\ref{xyz6}), we have
\begin{equation}
u+v \geq w+3,\label{uvw1a}
\end{equation}
\begin{equation}
v+w \geq u+3\label{uvw2a}
\end{equation}
and
\begin{equation}
w+u \geq v+3.\label{uvw3a}
\end{equation}
Suppose that $(u,v,w) \in P_1$. Then, there exist $(s,t,u)$ that is a permutation of $u,v,w$ and $(s,t,u) \in P_{1,1} \cup P_{1,2}$.

By (iii) of Lemma \ref{lemmaforn02},  for $(s,t,u) \in P_{1,1}$, we have 
\begin{equation}
|s+t-u| \leq 1. \label{uvwless1}
\end{equation}
Since $(s,t,u)$ is a permutation of $u,v,w$, 
the inequality in  (\ref{uvwless1}) contradicts (\ref{uvw1a}), (\ref{uvw2a}) and (\ref{uvw3a}).

By (iv) of Lemma \ref{lemmaforn02},  for $(s,t,u) \in P_{1,2}$, we have 
\begin{equation}
s+t=u+1. \label{splustequ1}
\end{equation}
Since $(s,t,u)$ is a permutation of $u,v,w$, 
the inequality in  (\ref{splustequ1}) contradicts (\ref{uvw1a}), (\ref{uvw2a}) and (\ref{uvw3a}).
\end{proof}

\begin{lemma}\label{notprio}
If we start with a position in $(x,y,z) \notin P$ and we reach $(u,v,w) \in N_0$,
then we can also reach $(p,q,r) \in P_1$
\end{lemma}
\begin{proof}
Suppose that 
\begin{equation}
(x,y,z) \notin P.\label{aussume}
\end{equation}
$[I]$ Suppose that there exists $w < z$ such that
$(x,y,w) \in N_{0,1} \cup N_{0,2}$ and $x, y \leq w$. Then,
we have
\begin{equation}
x \oplus y \oplus w =0, \label{case1nimsum}
\end{equation}
and we have the following (\ref{inn01}) and (\ref{inn01bb})
or (\ref{inn02}).
If $(x,y,w) \in N_{0,1}$, 
\begin{equation}
x+y=w \label{inn01}
\end{equation}
and 
\begin{equation}
x \geq 2. \label{inn01bb}
\end{equation}
If $(x,y,w) \in N_{0,2}$,
\begin{equation}
x+y=w+2. \label{inn02}
\end{equation}
$[I.1]$ Suppose that $w$ is even.
Then, by (\ref{inn01}) or (\ref{inn02}), $u+v$ is even, and 
$(x,y,w+1) \in P_1$. If $z=w+1$, then
we have $(x,y,z) \in P_{1,1}\cup P_{1,2}$, which contradicts 
(\ref{aussume}).
Then, $z > w+1$, and we reach $(x,y,w+1) \in P_1$.\\
$[I.2]$ Suppose that $w$ is odd.
Then, $x+y$ is odd, and 
$(x,y,w-1) \in P_1$.
If $y>w-1$, then we have $w=y$.
Then, by (\ref{case1nimsum}) $x=0$.
Here, these $x,y,w$ do not satisfy
(\ref{inn01bb}) and (\ref{inn01})
nor  (\ref{inn02}). Hence we have 
$y \leq w-1$.
Since $z > w-1 \geq y $, we reach $(x,v,w-1) \in P_{1,1} \cup P_{1,2}$.\\
$[II]$ Suppose that there exists $v < y$ that satisfy
$(x,v,z) \in N_0$,
\begin{equation}
x, v \leq z, \label{xvzorder}
\end{equation}
\begin{equation}
x \oplus v \oplus z =0, \label{xvznimsum}
\end{equation}
and the following (\ref{inn01b}) and (\ref{inn01c}) or (\ref{inn02b}).
If $(x,v,z) \in N_{0,1}$,
\begin{equation}
x+v=z \label{inn01b}
\end{equation}
and
\begin{equation}
x \geq 2. \label{inn01c}
\end{equation}
If $(x,v,z) \in N_{0,2},$
\begin{equation}
x+v=z+2. \label{inn02b}
\end{equation}
$[II.1]$ Suppose that $x,v,z$ are even.
By (\ref{xvznimsum}), we have
\begin{equation}
x \oplus (v+1) \oplus (z+1) = 0. \label{xv1z1nimsum0}
\end{equation}
If we have (\ref{inn01b}), 
\begin{equation}
x+(v+1)=z+1. \label{xv1eqz1}
\end{equation}
If we have (\ref{inn02b}),
\begin{equation}
x+(v+1)=z+1. \label{xv1eqz12}
\end{equation}
Since $x\geq 2$, by (\ref{xv1eqz1}) or (\ref{xv1eqz12}), $(x,v+1,z+1) \in N_{0,1} \cup N_{0,2}$, and 
$(x,v+1,z) \in P_{1,1} \cup P_{1,2}$. If $v+1=y$, $(x,y,z) \in P_1$. This contradicts 
$(\ref{aussume})$. Then, $v+1 < y$, and we can reach
$(x,v+1,z) \in P_1$.\\
Since $z > w+1$, we reach $(u,v,w+1) \in P_1$.\\
$[II.2]$ Suppose that $v,z$ are odd and $x$ is even.
Then, $(x,v-1,z-1) \in N_0$. Since $x+v-1$ is even,
$(x,v-1,z) \in P_1$ and we can reach $(x,v-1,z) \in P_{1,1} \cup P_{1,2}$.\\
Then, $x+v$ is odd, and 
$(x,v,z-1) \in P_1$.\\
$[II.3]$ Suppose that $x,z$ are odd and $v$ is even. 

First, we prove that 
\begin{equation}
v+1 \leq z-1. \label{v1lessz1}
\end{equation}
If (\ref{v1lessz1}) is not true, then
we have 
\begin{equation}
v=z \label{veqz}
\end{equation}
or
\begin{equation}
v+1=z. \label{v1eqz}
\end{equation}
If we have (\ref{veqz}),
by (\ref{xvznimsum}) $x=0$ and $x+v=z$.
These contradict (\ref{inn01c}) and (\ref{inn01b},
and these contradict (\ref{inn02b}), too.

If we have (\ref{v1eqz}),
there exists $n$ such that $v=2n$, $z=2n+1$. Since
$x \oplus v \oplus z = 0$, we have $x=1$, and hence
$x+v=z$.
These contradict (\ref{inn01c}) and (\ref{inn01b},
and these contradict (\ref{inn02b}), too.

Therefore, 
\begin{equation}
v+1 \leq z-1. \label{v1lesszm1}
\end{equation}

If $x=z$, by (\ref{xvznimsum}) $v=0$. Hence
$x+v=z$. Since $(x,v,z) \in N_{0,1}$, we have
$x,v \geq 2$. This leads to a contradiction.
Therefore, 
\begin{equation}
x <z. \label{xlessthanz}
\end{equation}

If $y=v+1$,
$(x,y,z-1) \in N_{0,2}$.
Since $x+y$ is even, $(x,y,z) \in P_{1,2}$, and this is a contradiction. Therefore,
\begin{equation}
y > v+1\label{ylargethanv1}
\end{equation}

By (\ref{v1lesszm1}), (\ref{xlessthanz}) and (\ref{ylargethanv1}),
we can reach $(x,v+1,z) \in P{1,1}.$

If we have (\ref{inn01b}), $x+(v+1)=(z-1)+2$.
Since $v+1 \leq z-1$, $(x,v+1,z-1) \in N_{0,2}$, and we can reach $(x,v+1,z) \in P_1$.

If we have (\ref{inn02b}), 
\begin{equation}
x+v=z+2. \label{inn03}
\end{equation}
By (\ref{xvznimsum}), (\ref{inn03}) and  Lemma \ref{lemmaforxyz}, $x,v$ are odd. Since
$v$ is even, we have a contradiction.\\
$[II.4]$ Suppose that $x,v$ are odd and $z$ is even.
Then,
\begin{equation}
 x_0=v_0=1 \text{ and }  z_0=0. \label{xvz0}
\end{equation}
$[II.4.1]$ If we have (\ref{inn01b}), by (i) Lemma \ref{lemmaforxyz}, we have 
$x_0 + y_0 = z_0$. This contradicts (\ref{xvz0}).\\
$[II.4.2]$ Suppose that we have (\ref{inn02b}).
If $x=1$, then by (\ref{inn02b}), $v=z+1$. This contradicts (\ref{xvzorder}).\\
If $v=1$, then by (\ref{inn02b}), $x=z+1$. This contradicts (\ref{xvzorder}).\\
If $x \ne 1$ and $v \ne 1$, then by (\ref{xvz0})
\begin{equation}
x,v \geq 3. \label{xvlarger3}   
\end{equation}
By (\ref{xvznimsum}),
\begin{equation}
  x \oplus (v-1) \oplus (z+1)=0  \label{xv1z11}
\end{equation}
and by (\ref{inn02b}),
\begin{equation}
x + (v-1) = z+1.\label{xvminus1z1}
\end{equation}
By (\ref{xvlarger3}), $x, v-1 \geq 2$, and
by (\ref{xv1z11}) and (\ref{xvminus1z1}),
\begin{equation}
(x,v-1,z+1) \in N_{0,1}.\label{xvminus1z1b}   
\end{equation}
Since $x+(v-1)$ is odd, by (\ref{xvminus1z1b})
$(x,v-1,z+1) \in P_{1,1}.$\\
$[III]$ Suppose that there exists $u < x$ such that
$(u,y,z) \in N_0$ and 
\begin{equation}
u, y \leq z \label{xvzorder3}
\end{equation}
Then, we prove the result of this theorem by the method that 
is similar to the one used in $[II]$.

\end{proof}
\begin{lemma}\label{notop1}
If we start with a position in $(x,y,z) \in N_0$ and we can reach $(u,v,w) \in P_1$.
\end{lemma}
\begin{proof}
$[I]$ Suppose that we start with $(x,y,z) \in N_{0,2}$.
Then, by $(ii)$ of Lemma \ref{lemmaforn02} and Definition \ref{defofpn},
\begin{equation}
x+y=z+2,\label{xplusyequalzplus2}    
\end{equation}
\begin{equation}
x \oplus y \oplus z = 0 \label{xyznim0forn02}    
\end{equation}
and $x,y$ are odd.
Then, 
\begin{equation}
x \oplus (y-1) \oplus (z+1) = 0  \label{xy1z1p1nim0}  
\end{equation}
and 
\begin{equation}
(x-1) \oplus (y-1) \oplus z = 0. \label{x1y1zp1nim0}  
\end{equation}
By (\ref{xplusyequalzplus2}), 
\begin{equation}
x+(y-1)=z+1,
\end{equation}
and by (\ref{xy1z1p1nim0}) 
$(x,y-1,z+1) \in N_{0,1}$.
Therefore, by  Definition \ref{defofpn}, $(x,y-1,z) \in P_{1,2}$, and 
we can move to $(x,y-1,z) \in P_{1,2}$.\\
$[II]$ Suppose that we start with $(x,y,z) \in N_{0,1}$.\\
Then, by Lemma \ref{lemmaforn02} and Definition \ref{defofpn},
$x+y=z$ and $x,y \geq 2$. Therefore, we can move to
$(x+y,0,z) \in P_{0,1}$ by amalgamation.
\end{proof}

\begin{theorem}
The set $P$ is the set of P-positions of the game of Definition \ref{defofamanimris}.
\end{theorem}
\begin{proof}
$[I]$ By lemmas \ref{p1notp0}, \ref{p0notp0}, \ref{p1notp1}, \ref{p01notp1}, and  \ref{p02notp1}, if we start with a position in $P$, we cannot reach a position in $P$.\\
$[II]$ Suppose that we start with a position $(x,y,z) \notin P$.\\
$[I.1]$ If $x \oplus y \oplus z \ne 0$,
we can reach a position $(u,v,w)$ such that $u \oplus v \oplus w =0$. Then,
we have $(u,v,w) \in P_0$ or $(u,v,w) \in N_0$.
If $(u,v,w) \in N_0$, then by Lemma \ref{notprio}, we can reach a position in $P_1$.\\
$[I.1]$ If $x \oplus y \oplus z = 0$, then $(x,y,z) \in N_0$. By Lemma \ref{notop1}, we can move to a position in $P$.
\end{proof}

\section{Prospect for Future Research}
I have the following conjecture. If this conjecture is true, the variant of amalgamation Nim in this article has an elegant mathematical structure.

\begin{conjecture}
For $n \in \mathbb{Z}_{\ge 0}$, we have the following $(i)$ and $(ii)$.\\
$(i)$ If the Grundy number of a position $(x,y,z)$ is $2n$, then 
$x \oplus y \oplus z = 2n$ or $x \oplus y \oplus z = 2n+1$.\\
$(ii)$ If the Grundy number of a position $(x,y,z)$ is $2n+1$, then 
$x \oplus y \oplus z = 2n$ or $x \oplus y \oplus z = 2n+1$.
\end{conjecture}

\vskip 20pt
\noindent {\bf Acknowledgement.}   I  extend my gratitude to Dr. Ryohei Miyadera for giving me instructions on the subject of this article.


\begin{thebibliography}{1}\footnotesize
\bibitem{lesson} M. H. Albert, R. J. Nowakowski and D. Wolfe, {\it Lessons In Play: An Introduction to Combinatorial Game Theory, Second Edition}, A K Peters/CRC Press, Natick, MA., United States, 2019.
\bibitem{bouton} C. L. Bouton, Nim, a game with a complete mathematical theory, {\it Ann Math.} {\bf 3 } (14) (1901-1902), 35--39.
\bibitem{levinenim} L. Levine, Fractal sequences and restricted nim, {\it Ars Combin.}  {\bf 80} (2006), 113--127.
\bibitem{amalnim} S.C. Locke and B. Handley, Amalgamation Nim, {\it Integers },  {\bf 21}, (2021), $\#$ G2.
\bibitem{richard}  D. Richards, Personal Communication, (2001).
\end{thebibliography}
\end{document}